\documentclass[reqno,11pt]{amsart}

\usepackage{amsfonts}
\usepackage{amsfonts}
\usepackage{mathrsfs}
\usepackage{graphicx}
\usepackage{amsfonts}
\usepackage[dvipsnames,usenames]{color}
\usepackage{amsmath, amsthm, amsfonts, amssymb}

\newlength{\figwidth}
\newlength{\figheight}
\newlength{\philwidth}

\def\mbb#1{\mathbb{#1}}

\def\lrs#1{\ensuremath{\left[{#1}\right]}}%
\def\lrv#1{\ensuremath{\lvert{#1}\rvert}}%

\newcommand{\1}{{\rm 1\hspace*{-0.4ex}%
\rule{0.1ex}{1.52ex}\hspace*{0.2ex}}}
\newcommand{\R}{\ensuremath{\mathbb{R}}}

\newcommand{\NN}{{\mathbb N}}

\newcommand{\abs}[1]{\lrv{#1}}%

\def\Prblr#1{\mbb{P}\kern-2pt\lrs{{#1}}}%

\newcommand{\prob}[1]{{\mathbb P}\Big(#1\Big)}
\newcommand{\expect}[1]{{\mathbb E}\left(#1\right)}

\newcommand{\lc}{\left\{}
\newcommand{\rc}{\right\}}
\newcommand{\lb}{\left(}
\newcommand{\rb}{\right)}
\newcommand{\ls}{\left[}
\newcommand{\rs}{\right]}

\newcommand{\re}{{\mathbb R}}

\newcommand{\tdvarphi}{\tilde{\varphi}}

\newcommand{\tp}{\tilde{\psi}}
\newcommand{\Wc}{W^o}

\newtheorem{lemma}{Lemma}
\newtheorem{theorem}{Theorem}

\definecolor{darkorange}{rgb}{1.00,0.50,0.00}
\begin{document}
\title[Asymptotic Results for Random Polynomials on the Unit Circle]{Asymptotic Results for Random Polynomials on the Unit Circle}

\author{Gabriel H. Tucci and Philip A. Whiting}
\address{Gabriel H. Tucci and Philip A. Whiting are with Bell Labs, Alcatel--Lucent, 600 Mountain Ave, Murray Hill, NJ 07974.}
\email{gabriel.tucci@alcatel-lucent.com}
\email{philip.whiting@alcatel-lucent.com}

\begin{abstract}
In this paper we study the asymptotic behavior of the maximum magnitude of a complex random polynomial with i.i.d. uniformly distributed random roots
on the unit circle. More specifically, let $\{n_k\}_{k=1}^{\infty}$ be an infinite sequence of positive integers and let
$\{z_{k}\}_{k=1}^{\infty}$ be a sequence of i.i.d. uniform distributed random variables on the unit circle.
The above pair of sequences determine a sequence of random polynomials $P_{N}(z) = \prod_{k=1}^{N}{(z-z_k)^{n_k}}$ with random roots on the unit circle and their corresponding
multiplicities. In this work, we show that subject to a certain regularity condition on the sequence $\{n_k\}_{k=1}^{\infty}$, the
log maximum magnitude of these polynomials scales as $s_{N}I^{*}$ where $s_{N}^{2}=\sum_{k=1}^{N}{n_{k}^{2}}$ and $I^{*}$ is a strictly positive random variable.
\end{abstract}

\maketitle

\section{Introduction}\label{intro}

Random polynomials are ubiquitous in several areas of mathematics and have found several applications in diverse fields such as random matrix theory, representation theory and chaotic systems (see \cite{kac, berry, shepp, edelman, shme,bleher}). The geometric structure of random polynomials is of significant interest as well. Constructing a random polynomial from its roots is a natural construction which can be expected to occur in a wide range of settings. For instance, consideration of the asymptotic behavior of the maximum magnitude allowed the authors to obtain a lower bound on the minimum singular value for random Vandermonde matrices (see \cite{JOTP} for more details). 

In this work we continue the investigation of such polynomials, where we allow for
non--constant multiplicity of the roots. We show that providing this sequence satisfies a simple
sufficient condition, the limit distribution of the maximum magnitude on
the unit circle is determined by a certain Gaussian process obtained from the Brownian bridge. Surprisingly this limit, up to renormalization, does not depend on the 
sequence itself. 

Our construction for these random polynomials is as follows. Let $\{n_k\}_{k=1}^{\infty}$ be an infinite sequence of positive integers and let
$\{z_{k}\}_{k=1}^{\infty}$ be sequence of i.i.d. uniform distributed unit magnitude complex numbers.
The above pair of sequences, then determine a sequence of random polynomials $P_{N}(z) = \prod_{k=1}^{N}{(z-z_k)^{n_k}}$ with roots on the unit circle and their corresponding
exponents. Our main result relies on a construction based on the Brownian bridge and enables us to conclude that  
\begin{equation}
\frac{1}{s_N} \log \max \Big\{ |P_N(z)|^2\,\,:\,\,|z|=1\Big\} \Rightarrow I^{*}  
\end{equation}
converges weakly to a positive random variable $I^{*}$ where $s_{N}^2=\sum_{k=1}^{N}{n_k^2}$.
 
\section{Random Polynomials}\label{sec_randpoly}

\subsection{Pointwise Convergence and Lindberg Condition}\label{sec_prelim}

Let $P_{N}$ be as before and let $L_{N}(\psi)$ be defined as 
$$
L_N(\psi) := \log |P_{N}(e^{i\psi})|^{2} = \sum_{k=1}^N n_k \log \Big(2 \lb 1 - \cos ( \psi - \theta_k) \rb \Big)
$$
for $\psi\in[0,2\pi]$.  Also let $s_N^2 := \sum_{k=1}^{N} n_k^2$ and $T_N(\psi) := L_N(\psi)/s_N$.

As explained in the introduction we are interested in the behaviour of the maximum magnitude squared of $P_N(z)$ as $N$ increases. Since $\abs{P_N(z) }^2$ is a continuous
function on the  unit circle, it follows that there exists $\varphi^*$ that attains its maximum. Let $T_{N}^{*}$ be this value. For later use, let 
$\Phi:=\{\varphi_r\,\,:\,\,r\geq 0\}$ be the set of $2\pi$ times the dyadic rationals on the interval $[0,1]$. Then it is clear that
$$
\limsup_{r \rightarrow \infty} {\,T_N(\varphi_r)} = T_N^*.
$$
The case $n_k=1$ appears in connection with the asymptotic behaviour of the minimum eigenvalue of random Vandermonde matrices (see 
\cite{TW, JOTP} for more details). In this special case, weak convergence to the normal distribution holds
$$
T_N(\psi)  \Rightarrow N(0,\sigma^2)
$$
for every fixed $\psi$ where 
$$
\sigma^2 := \frac{1}{2\pi}\int_0^{2\pi} \log^2 \big(2 ( 1 - \cos \psi)\big) d\psi \approx  3.292.
$$
This is in fact a consequence of the central limit theorem since 
$\log \big(2 ( 1 - \cos \psi)\big)$ is square summable and $\int_0^{2\pi} \log \big(2 ( 1 - \cos \psi)\big) d\psi = 0$.

In what follows we derive a simple sufficient condition for the asymptotic normality of the random variable $T_N$. We use the notation $\expect{X;A} := \expect{X \1_A}$ where $X$ is a random variable and $A$ is a Borel set. Let $X_k$ be a sequence of independent zero mean and variance $\sigma^2_k$ random variables. We say that this sequence satisfies the Lindberg condition \cite{Feller_Vol2} if and only if
\begin{equation}
\lim_{N\to\infty} \frac{1}{s_N^2} \sum_{k=1}^N \expect{ X_k^2\,;\, \abs{ X_k} \geq \epsilon s_N} = 0
\end{equation}
for every $\epsilon > 0$. If $X_k = \sigma_k Y_k$ then this condition becomes
\begin{equation}
\lim_{N \rightarrow \infty}  \frac{1}{s_N^2} \sum_{k=1}^N \sigma^2_k \expect{ Y^2 \,;\, \abs{ Y} \geq \frac{\epsilon s_N}{\sigma_k}} = 0.
\label{eqn_easyLind}
\end{equation}

For our purposes we only focus on the case where the random variables $Y_{k}$ are i.i.d. according to the distribution of $Y = \log \big(2 (1- \cos (2\pi U))\big)$ and where $U$ is uniform random variable on $[0,1]$. The main result of this section is the following Theorem.

\begin{theorem}[Lindberg Exponent]
\label{thm_Lindberg}
If 
\begin{equation}
\lim_{N\to \infty} \sum_{k=1}^N e^{-\frac{\epsilon s_N}{n_k}} = 0
\label{eqn_Suff}
\end{equation}
holds for every $\epsilon > 0$ then  
\begin{equation}
T_{N}(\psi)=\frac{L_N(\psi)}{s_N} \Rightarrow N(0,I_\sigma).
\end{equation}
\end{theorem}

\begin{proof}
By definition  $Y = \log  2 (1- \cos 2\pi U)$, where $U$ is uniform on $[0,1]$,
from which it follows that $Y \leq \log 2(1 -\cos \pi) = \log 4$ and hence the
moment generating function exists for all $t > 0$. 

Additionally, since $s_N,n_k$ and $\epsilon$ are positive it follows that,
\begin{equation}
\expect{Y^2 ; Y \geq \frac{s_N\epsilon}{n_k}} \leq \log^{2}(4) \prob{Y \geq \frac{s_N\epsilon}{n_k}} 
\label{eqn_posY}
\end{equation}
and applying Markov's inequality with $t > 0$ we obtain 
\begin{equation}
\prob{Y \geq \frac{s_N\epsilon}{n_k}} \leq 4^t e^{-\frac{t s_N\epsilon}{n_k}}
\label{eqn_Yright}
\end{equation}
We thus obtain the following bound for the upper Lindberg condition,
\begin{eqnarray*}
\frac{1}{s^2_N} \sum_{k=1}^N n_k^2 \expect{Y^2 ; Y \geq \frac{s_N\epsilon}{n_k}}
& \leq & \frac{ \log^{2}(4)  4^t }{s^2_N} \sum_{k=1}^N n_k^2 e^{-\frac{t s_N\epsilon}{n_k}}\\
& \leq & \log^{2}(4) 4^t  \sum_{k=1}^N e^{-\frac{t s_N\epsilon}{n_k}}.
\end{eqnarray*}
The sum on the RHS goes to 0 as this is the condition we assumed holds and because $t$ and $\epsilon$ are arbitrary and fixed. We now turn to the lower Lindberg condition, where $Y \leq  -\frac{s_N\epsilon}{n_k}$. 
Let $\xi_{N,k} := \epsilon s_N/n_k$, and $\theta$ in the interval $[-\pi/3,\pi/3]$. Using the following two basic inequalities, 
\begin{equation}
\label{eqn_cossquare}
\cos \theta \leq 1 - \theta^2/4
\end{equation}
for $0 \leq \abs{\theta} \leq \pi/2$ and, 
\begin{equation}
\log^{2} \big(2 \lb 1 - \cos \theta \rb\big) \leq \log^{2} \big(\theta^2/2\big)
\end{equation}
for $0 \leq \abs{\theta} \leq \pi/3$ we obtain that,
\begin{eqnarray*}
\expect{Y^2; Y \leq -\xi_{N,k} }
& \leq & \expect{ \log^{2} \big(\theta^2/2\big)\,; Y \leq -\xi_{N,k} } \\
& \leq & \expect{ \log^{2} \big(\theta^2/2\big)\,; E_{N,k} }
\end{eqnarray*} 
where $E_{N,k} = \lc \abs{\theta} \leq \sqrt{2} e^{-\xi_{N,k}/2} \rc$. 
Since   $\int \log^{2}(x) dx = x \log^{2}(x)  - 2x \log(x) + 2x$, we make the substitution
$\phi = \theta/\sqrt{2}$ and determine the above expectation to be
\begin{eqnarray*}
\expect{\log^{2}(\phi^2) ; \abs{\phi} \leq \delta} &= &\frac{8}{\pi\sqrt{2}} h(\delta)\\
\end{eqnarray*}
where $h(\delta) := \delta\log^{2}(\delta) - 2 \delta \log(\delta) + 2\delta $ 
and $\delta := e^{-\xi_{N,k}}$. Rewriting this expression we obtain 
$$
\frac{4\sqrt{2}}{\pi}e^{-\epsilon s_N/(2n_k)} \lb \frac{\epsilon^2 s_N^2}{4n^2_k}
+ \frac{\epsilon s_N}{n_k} + 2 \rb.
$$
Summing over $k$ and dividing by $s_N^2$ we obtain that 
$$
 \frac{4\sqrt{2}}{\pi}\sum_{k=1}^N e^{-\epsilon s_N/(2n_k)} \left(
 \frac{\epsilon^2}{4} + \frac{\epsilon n_k}{s_N} + 2\frac{n^2_k}{s^2_N} \right) \rightarrow 0.
$$
Now applying the classical Lindberg's Theorem \cite{Billingsley} the proof is complete.
\end{proof}

The condition is easily verified to hold when $n_k = k^{p}$ for $p\geq 0$. Similarly, it is easy to show that this condition fails if 
$n_k = 2^k$.

\subsection{Finite Dimensional Limits}
The result in Section \ref{sec_prelim} was for the marginal distribution of $T_N(\psi)$ for a fixed value $\psi$. However, we would like to consider the weak limit for the sequence $T_{N,r}$ where $T_{N,r} := T_{N}(\phi_r)$. It is well known that weak convergence in the sequence space $\R^\infty$ is entailed by weak 
convergence of the finite dimensional distributions. For this reason, it is important to understand the joint distribution of $s$ such variables. We first focus on the case $s=2$. Define the covariance function
\begin{equation}
K(\theta) := (2\pi)^{-1} \int_0^{2\pi}\log \big(2\lb 1 - \cos \psi \rb\big) \log \big(2 \lb 1 - \cos \lb \psi + \theta \rb \rb \big) d\psi 
\end{equation}
for $\theta\in [0,2\pi]$. The plot of this function is shown in Figure \ref{fig_Kfun}. 
\begin{figure}[!Ht]\label{fig_Kfun}
  \begin{center}
    \includegraphics[width=10cm]{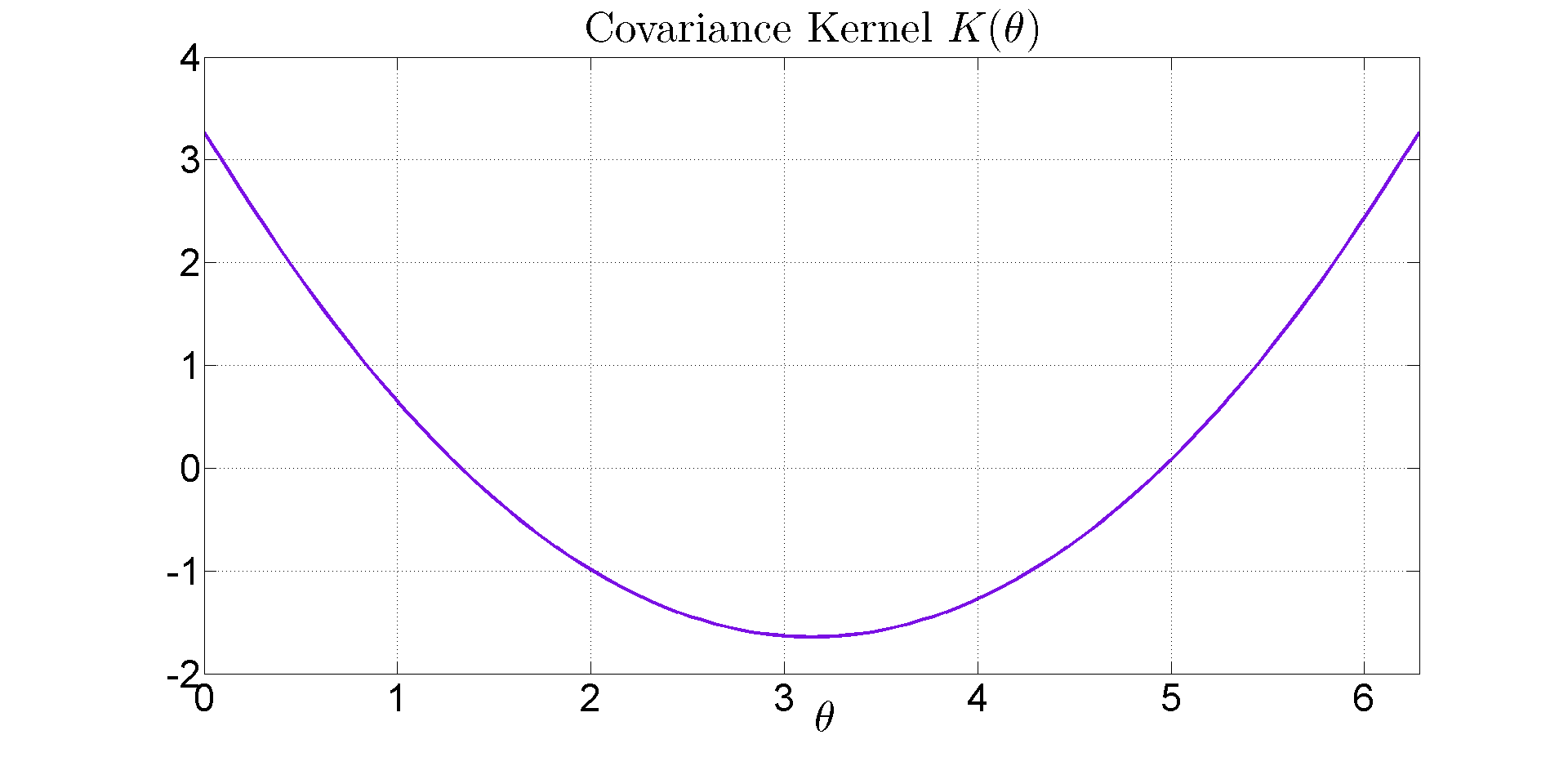}
    \caption{Graph of the covariance function $K(\theta)$.}
  \end{center}
\end{figure}
Given $t$ and $u \in \R$ define,
\begin{eqnarray}
V_N & := &  t T_N(\varphi_1) + u T_N(\varphi_2)  \label{eqn_Veqn} \\
     & = & \frac{t}{s_N} \sum_{\ell=1}^N n_{\ell}\log \abs{e^{i\varphi_1} - e^{i\theta_\ell}}^2 + \frac{u}{s_N}\sum_{\ell=1}^N n_{\ell}\log \abs{e^{i\varphi_1}-e^{i\theta_\ell}}^2 \nonumber \\
     & = & \frac{1}{s_N} \sum_{\ell=1}^N V_{\ell,N} \nonumber    
\end{eqnarray}    
which is a scaled sum of zero mean i.i.d. random variables with variance,
$$
\expect{V_{\ell,N}^2} = n_{\ell}\Big(\lb t^2 + u^2 \rb \sigma^2+ 2tu K(\abs{\varphi_1-\varphi_2})\Big).
$$
In what follows we denote denote $V = t W + u Z$ where $W = T_N(\varphi_1)$ and  $Z = T_N(\varphi_2)$. We assume
that $t \neq 0$ and $u \neq 0$ since otherwise there is nothing to show. It is not difficult to see that 
\begin{eqnarray*}
V^2 & =  & t^2W^2 + 2tu WZ + u^2Z^2 \\
  & \leq & t^2 W^2 + \abs{tu}(W^2 + Z^2) + u^2 Z^2. \\
\end{eqnarray*}
However, the triangle inequality implies that if $\abs{V} \geq \epsilon s_N/n_k$ then either, $\abs{W} \geq  \epsilon s_N/(2\abs{t} n_k)$
or the corresponding inequality for $Z$ holds (or both). However, we have 
already demonstrated that
$$
\frac{1}{s_N^2}\sum_{k=1}^N n_k^2\expect{W^2 ; \abs{W} \geq  \frac{\epsilon s_N}{2\abs{t} n_k}} \rightarrow 0
$$
as $N \rightarrow\infty$. As far as the terms involving $Z^2$ are concerned we
only need to show the above in the case that,
$$
\abs{Z} \leq \frac{\epsilon s_N}{2\abs{u} n_k}
$$
that is we wish to show,
$$
\frac{1}{s_N^2}\sum_{k=1}^N n_k^2\expect{Z^2 ; \abs{W} \geq  \frac{\epsilon s_N}{2\abs{t} n_k}, \abs{Z} \leq  \frac{\epsilon s_N}{2\abs{u} n_k}} \rightarrow 0
$$
but the above is smaller than
\begin{equation}
\frac{\epsilon^2}{4 u^2} \sum_{k=1}^N \prob{\abs{W} \geq  \frac{\epsilon s_N}{2\abs{t} n_k}}.
\label{eqn_finsum}
\end{equation}
Using Markov's inequality as in (\ref{eqn_Yright}) we see that
$$
\prob{ W \geq  \frac{\epsilon s_N}{2\abs{t} n_k}} \leq 4^\eta e^{-\eta \frac{\epsilon s_N}{2\abs{t} n_k}}.
$$
Applying (\ref{eqn_cossquare}) and using that $\theta \sim U[0,2\pi]$ we see that
$$
\prob{ W \leq -\frac{\epsilon s_N}{2\abs{t} n_k}} \leq \frac{2}{\pi} e^{-\frac{\epsilon s_N}{4\abs{t} n_k}}.
$$
It follows that (\ref{eqn_finsum})  tends to 0 as $N \rightarrow \infty$.

Thus by an extension of the arguments given in the proof of Theorem \ref{thm_Lindberg} it can be shown that condition (\ref{eqn_Suff}) is sufficient for the Lindberg condition to hold in respect of the
random variables $V_{N,\ell}$. Therefore, $\lb T_N(\varphi_1), T_N(\varphi_2)\rb \Rightarrow N(0,\Sigma)$ with $\Sigma_{11} = \Sigma_{22} = \sigma^2$ and $\Sigma_{12} = \Sigma_{21} = K(\abs{\varphi_1-\varphi_2})$ as an application of the Cramer--Wold device \cite{Billingsley68}.

Clearly the above arguments go through in the case of 3 or more variables. Therefore, the following result holds.
\begin{theorem}
\label{thm_Kvariate}
Let $(\varphi_1,\ldots,\varphi_s)$ be $s$ numbers in $[0,2\pi]$ and let $\lb T_N(\varphi_1),\ldots,T_N(\varphi_s) \rb$ be the corresponding random
vector. Then,
$$
\lb T_N(\varphi_1),\ldots,T_N(\varphi_s) \rb  \Rightarrow N(0, \Sigma_s)
$$
i.e. asymptotically joint normal with covariance determined by 
$\Sigma_{k,\ell} = K(|\varphi_k - \varphi_\ell|)$ where 
$$
K(\theta) := (2\pi)^{-1} \int_0^{2\pi}\log \big(2\lb 1 - \cos \psi \rb\big) \log \big(2 \lb 1 - \cos \lb \psi + \theta \rb \rb\big) d\psi.
$$  
\end{theorem}

\section{Asymptotic Distribution of $T_N^*$}

The following is an alternative way to construct the limit distribution of the random sequence
$T_{N,r}$. Given a realization of the Brownian bridge $\Wc$ on $[0,2\pi]$ (which satisfies
$\Wc(0) = \Wc(2\pi) = 0$). A $\varphi$ shift of the Brownian bridge is defined as
\begin{equation*}
\Wc_\varphi(\theta) :=
\begin{cases}
\Wc(\varphi+\theta) - \Wc(\varphi) & \theta \in [0,2\pi -\varphi],\\
\Wc(\varphi+\theta-2\pi) - \Wc(\varphi) & \theta \in [2\pi - \varphi,2 \pi]. 
\end{cases}
\end{equation*}
In addition, define the function $I:[0,2\pi] \rightarrow \R$ by
$$
I_\varphi :=  \int_0^{2\pi} \Wc_\varphi(\theta) \frac{\sin \theta}{1 - \cos \theta} d \theta
$$
for $\varphi\in[0,2\pi]$. Figure 2 shows us a realization of $I_\varphi$.
\begin{figure}[!Ht]\label{fig_itt}
  \begin{center}
    \includegraphics[width=12cm]{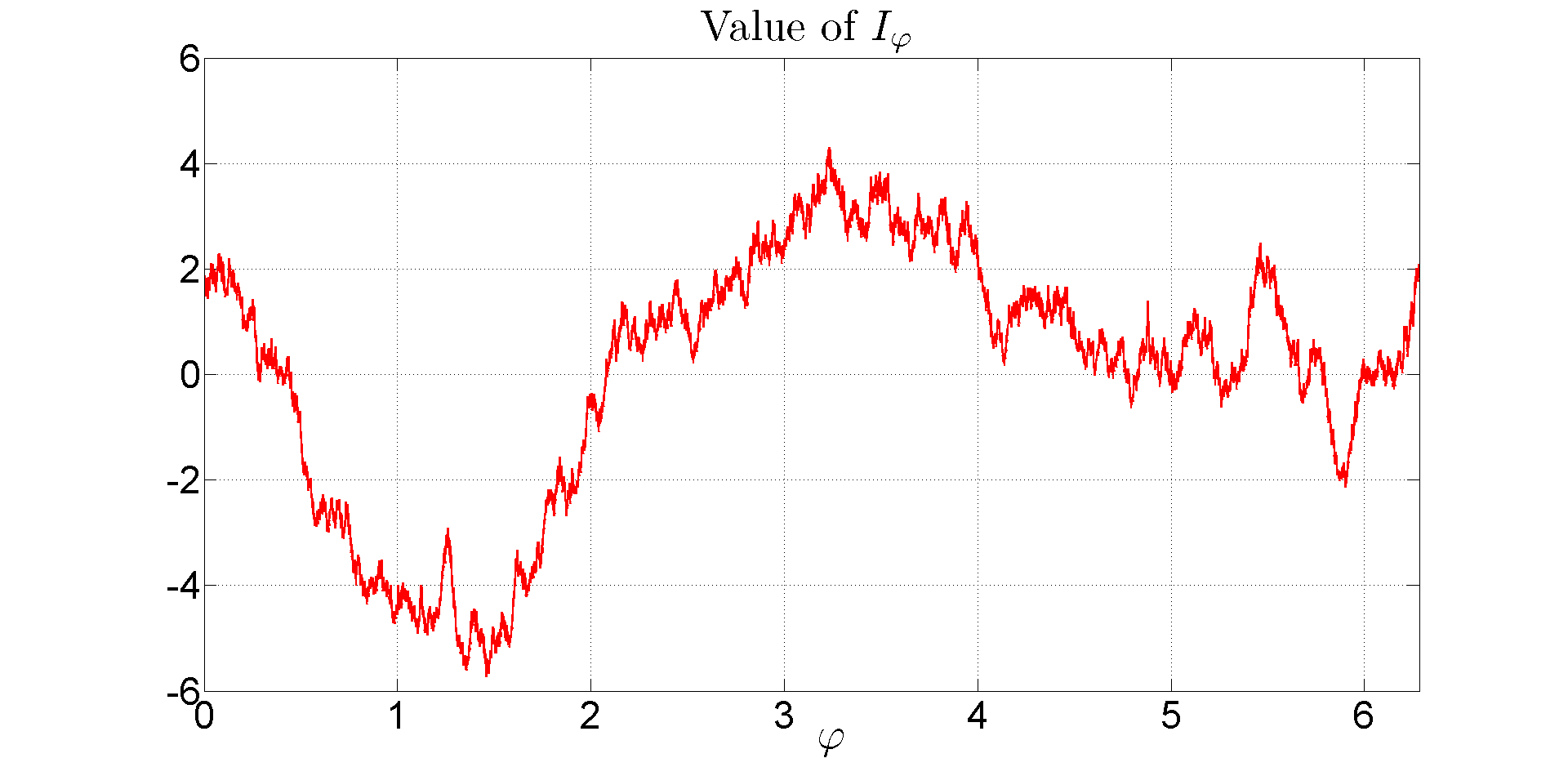}
    \caption{$I_\varphi$ for a realization of the Brownian bridge.}
  \end{center}
\end{figure}
The following Lemma shows that $I$ is a well defined quantity almost surely. 

\begin{lemma}\label{lemma_ctsI}
Given a realization of the Brownian bridge $\Wc$, then a.s. the following integral exists for all $\varphi \in [0, 2\pi)$
$$
|I_\varphi|=\Bigg|\int_0^{2\pi} \Wc_\varphi \frac{\sin \psi}{1 - \cos \psi} d \psi \Bigg| < \infty.
$$
In addition, the function $\varphi\mapsto I_\varphi$ is continuous. 
\end{lemma}

\begin{proof}
When $\varphi=0$ we write the above integral as $I$. The Levy global modulus of continuity tells us that for standard Brownian motion $B$ on $[0,2\pi)$
$$
\lim_{\delta\to 0} \limsup_{0 \leq t \leq 2 \pi - \delta} \frac{\abs{B(t+\delta) - B(t)}}{w(\delta)} = 1
$$
where $w(\delta) = \sqrt{2 \delta \log \frac{1}{\delta}}$ (see \cite{RogersWilliams} for a proof of this result). Since $\Wc$ is by definition,
$$
\Wc(\psi) = B(\psi) - \frac{\psi}{2 \pi} B(2 \pi)
$$
our argument is the same no matter which value of $\varphi$ is chosen because the Levy modulus applies to the entire sample path. We therefore set $\varphi = 0$. By definition of the Levy modulus, there exists $\delta_2 > 0$ almost surely  such that
$$
\frac{\abs{B(t+\delta) - B(t)}}{w(\delta)} \leq 2
$$
for all $0<\delta\leq\delta_2$.
Therefore,
\begin{equation}
a(\delta):=\abs{\Wc(\psi+\delta) - \Wc(\psi)} \leq 2 w(\delta) + \frac{|B(2\pi)|}{2\pi} \delta.
\end{equation}
We may therefore split the integral as,
\begin{equation}
I  =  \int_{\delta_2}^{2\pi- \delta_2} \Wc(\psi) \frac{\sin \psi}{1 - \cos \psi} d \psi + \int_0^{\delta_2} \Wc(\psi)\frac{\sin \psi}{1 - \cos \psi} d \psi
 + \int_{2 \pi - \delta_2}^{2\pi} \Wc(\psi) \frac{\sin \psi}{1 - \cos \psi} d \psi.
\end{equation}
The first integral is finite being the integral of a continuous function over the interval $[\delta_2, 2\pi - \delta_2]$. We may further suppose that $\delta_2$ has been chosen so that 
$\abs{\psi \frac{\sin \psi}{1 - \cos \psi}} \leq 4$ for $0 < \psi < \delta_2$ with the corresponding inequality in a similar neighbourhood of $2 \pi$. By choice of $\delta_2$ we obtain that  
$$
\Big| \int_0^{\delta_2} \Wc(\psi) \frac{\sin \psi}{1 - \cos \psi} d \psi \Big|
\leq  8 \int_0^{\delta_2} \frac{a(\psi)}{\psi} d \psi = O(\delta_2^{1/3})
$$ 
for sufficiently small $\delta_2$. The same argument applies to the last integral. Since $w(\delta_2)$ gives
a uniform bound the result holds for all $\varphi \in [0, 2 \pi)$. Continuity in $\varphi$ follows by a similar argument, 
\begin{eqnarray}
\vert I_\varphi - I_{\tdvarphi} \vert & \leq & \Big| \int_\delta^{2\pi- \delta} \lb \Wc_\varphi - \Wc_{\tdvarphi} \rb
\frac{\sin \psi}{1 - \cos \psi} d \psi \Big| \\
& + & \int_0^{\delta} \abs{\Wc_\varphi(\psi)}  \Big|\frac{\sin \psi}{1 - \cos \psi}\Big| d \psi \nonumber \\
 & + & \int_{2\pi - \delta}^{2\pi} \abs{\Wc_\varphi(\psi)} \Big|\frac{\sin \psi}{1 - \cos \psi}\Big| d\psi \nonumber \\
 & + & \int_0^{\delta} \abs{\Wc_{\tdvarphi}(\psi)} \Big|\frac{\sin \psi}{1 - \cos \psi}\Big|  d\psi \nonumber \\
 & + &  \int_{2\pi - \delta}^{2\pi} \abs{\Wc_{\tdvarphi}(\psi)} \Big|\frac{\sin \psi}{1 - \cos \psi}\Big| d\psi. \nonumber 
\end{eqnarray}

Provided that $0 < \delta < \delta_2$, the tail integrals are all at most $O(\delta^{1/3})$ as before. We bound the first integral by two positive integrals, to obtain
\begin{eqnarray*}
\Big|\int_\delta^{2\pi- \delta} \lb \Wc_\varphi - \Wc_{\tdvarphi} \rb \frac{\sin \psi}{1 - \cos \psi}d \psi \Big|
& \leq & 2 \sup \,\abs{\Wc_{\varphi}(\psi) - \Wc_{\tdvarphi}(\psi)}  \int_\delta^\pi \frac{\sin \psi}{\lb 1 - \cos \psi \rb} d\psi \\ 
& \leq &  2 \sup \,\abs{\Wc_\varphi(\psi) - \Wc_{\tdvarphi}(\psi)}  \big[ \log\lb 1 - \cos \psi \rb \big]_\delta^\pi \\
& \leq & 6 a(\delta) \lb \log 2 - \log (1 - \cos \delta) \rb
\end{eqnarray*}
provided $\abs{\varphi - \tdvarphi} < \delta$. Finally, since 
$$
a(\delta) \lb \log 2 - \log \big(1 - \cos \delta\big) \rb \to 0
$$ as $\delta \rightarrow 0$ we finish the proof.
\end{proof}

Let $\Phi=\{\varphi_r \,\,:\,\, r\geq 0\}$ be the sequence described in Section \ref{sec_prelim} and let ${\bf I}=\{I_{r}\}_{r=0}^{\infty}$ be the sequence defined as $I_r :=  I_{\varphi_r}$. 
Since the function $I_\varphi$ is continuous on the interval $[0,2\pi]$ there exists a value $\varphi^*$ which determines the maximum value of  $I_\varphi$, which we denote by $I^*$. Since $\Phi$ is dense on the unit circle it follows that  
\begin{equation}\label{eqqmax}
I^{*} := \sup \lc  I_r : r \in \NN \rc
\end{equation}
and its distribution is determined via the infinite sequence $I_r$. Let the sequence of random variables ${\bf T}_{N}=\{T_{N,r}\}_{r=0}^{\infty}$. We now derive one more Lemma for use later on.
\begin{lemma}
\label{lem_Dctsmap}
Let $Y$ be a function in $D[0,2\pi]$. Then $Y$ is Lebesgue measurable, and its integral exists,
\begin{equation}
\int_0^{2\pi} Y(s) ds < \infty.
\end{equation}
Furthermore, let $Y_n$ be a sequence of functions in $D[0,2\pi]$ such that 
$Y_n \rightarrow Y$ in $D$ (i.e. with respect to the Skorohod topology) then
\begin{equation}
\int_0^{2\pi} Y_n(s) ds \rightarrow \int_0^{2\pi} Y(s)ds.
\end{equation}
\end{lemma}

\begin{proof}
The existence of the integral follows from Lemma 1, page 110 of \cite{Billingsley68} and the subsequent discussion
which shows that functions in $D$ on a closed bounded interval are both Lebesgue measurable and bounded.
The former follows from the fact that they can be uniformly approximated by simple functions, a direct
consequence of Lemma 1 and the latter also. 

Convergence follows from the Lebesgue dominated convergence theorem.
This holds since the sequence $Y_n$ is uniformly bounded, by a constant so the sequence is dominated.
Second $Y$ is continuous a.e. with pointwise convergence holding at points of continuity, as a consequence
of convergence in $D$ see \cite{Billingsley68}.
\end{proof}

We now proceed to prove the following Theorem.

\begin{theorem}\label{thm_weakconvergence}
The sequence ${\bf T}_{N}$ converges in distribution to the sequence ${\bf I}$
\begin{equation}\label{eqn_thmfin}
{\bf T}_N \Rightarrow {\bf I}
\end{equation}
as $N\to\infty$.
\end{theorem}

\begin{proof}
In order to do so we use Theorem 4.2 of \cite{Billingsley68}. Suppose that there is a metric space ${\mathcal S}$ with metric $\rho_0$ and sequences  
${\bf T}_{N, \epsilon}$, ${\bf I}_\epsilon$ and ${\bf T}_N$ all lying in ${\mathcal S}$ such that the following conditions hold,
\begin{eqnarray}\label{eqn_weakcond} 
 {\bf T}_{N, \epsilon} & \Rightarrow & {\bf I}_\epsilon \\
 {\bf I}_\epsilon & \Rightarrow & {\bf I} \nonumber
\end{eqnarray}
together with the further condition that given arbitrary $\eta > 0$,
\begin{equation}\label{eqn_probmetric}
\lim_{\epsilon \rightarrow 0} \limsup_{N \rightarrow \infty} 
\prob{ \rho_0({\bf T}_{N, \epsilon}, {\bf T}_N) \geq \eta } = 0.
\end{equation}

Then it holds that ${\bf T}_N \Rightarrow {\bf I}$. First, we define ${\bf I}_\epsilon$ using a realization of the Brownian bridge as follows,
$$
I_{r,\epsilon} := \big[ \Wc_{\varphi_r}(\psi) \log 2(1- \cos \psi) \big]^{2\pi - \epsilon}_{\epsilon}
- \int_\epsilon^{2\pi-\epsilon} \Wc_{\varphi_r}(\psi)  \frac{\sin \psi}{1 - \cos \psi} d \psi.
$$
The definition of the other sequence is more involved and so we defer it for a moment. We have shown that the limit integrals exist a.s. and so we only need to show that the 
first term converge to 0. Since $\log \big(2 \lb 1 - \cos \psi \rb\big) = O(\log \epsilon)$ 
when $\epsilon$ is small and in a neighbourhood of 0 and $2\pi$ we may invoke the 
Levy modulus of continuity, wrapped around at $2\pi$ to obtain that this term is
$$
O( \log \epsilon a(\epsilon)) \rightarrow 0.
$$
Hence, coordinate convergence of the integrals holds so that
$$
\Big| I_{r,\epsilon} + \int_\epsilon^{2\pi-\epsilon} \Wc_{\varphi_r}(\psi)  \frac{\sin \psi}{1 - \cos \psi} d \psi \Big| \Rightarrow 0
$$
and it follows that ${\bf I}_\epsilon \Rightarrow {\bf I}$ as $\epsilon\to 0$, since the sign of the integral is immaterial. We have thus demonstrated the second condition of (\ref{eqn_weakcond}). Next, we proceed by rewriting $T_N(\varphi_r)$ in terms of the empirical distribution function $F_N:[0,2\pi]\to [0,1]$ determined by
$$
F_N(\psi) := \frac{\# \lc \theta_q : 0\leq \theta_q \leq \psi \rc}{s_N}.
$$
By definition of $F_{N}(\psi)$ and the Lebesgue--Stieljes integral we see that
\begin{eqnarray*}
T_N(\varphi_r) & = & \sqrt{s_N} \int_0^{2\pi} \log \big(2(1 - \cos(\varphi_r-\psi))\big) dF_N(\psi) \\
               & = & \sqrt{s_N} \int_0^{2\pi} \log \big(2(1 - \cos \tp)\big) dF_{N,\varphi_r}(\tp) 
\end{eqnarray*}
where the change of variables, $\tp = \psi - \varphi$ has been made.  For $\psi \in [0,2\pi)$ we define $F_{N,\varphi}(\psi)$ as the ``cycled'' empirical distribution function of $F_{N}$ by

\begin{equation*}
F_{N,\varphi}(\psi) :=
\begin{cases}
\frac{\#\lc \varphi \leq \theta_q < \varphi+\psi \rc }{s_N} & \text{if } \varphi \leq \psi < 2\pi-\varphi, \\ 
F_{N,\varphi}(2\pi) + \frac{\#\lc 0 \leq \theta_q \leq \psi-2\pi+\varphi \rc}{s_N}& \text{if } 2\pi-\varphi \leq \psi < 2\pi.
\end{cases}
\end{equation*}

To define the sequence $T_{N,\epsilon}(\varphi_r)$ we split the integral into two parts 
as in $\int^{2\pi - \epsilon}_\epsilon$ and $\int_0^\epsilon + \int_{2\pi - \epsilon}^{2\pi}$
and then use  integration by parts on the first part, which yields the expression,
\begin{eqnarray}
 T_{N,\epsilon}(\varphi_r) & := & 
\sqrt{s_N} \lb \ls \lb F_{N,\varphi_r}(\psi) - \frac{\psi}{2\pi} \rb \log 2(1 - \cos \psi) \rs_\epsilon^{2\pi-\epsilon} \rb \\
\label{eqn_TNepsdefn}
 & - & \sqrt{s_N}  \int_\epsilon^{2\pi - \epsilon} \lb F_{N,\varphi_r}(\psi) - \frac{\psi}{2\pi} \rb\frac{\sin \psi}{(1 - \cos \psi)} d\psi \nonumber .
\end{eqnarray}
For later use we make the definition,
$$
W_{N,\varphi} := \sqrt{s_N} \lb  F_{N,\varphi}(\psi) - \frac{\psi}{2\pi} \rb.
$$
This is not quite equal to the original sum, since
$$
\int_0^{2\pi}  \log \big(2 \lb 1 - \cos \psi \rb\big) d \psi = 0
$$
so that the $\psi$ terms do not give 0 but rather cancel with $\mu_\epsilon$ to be defined in  a moment. The remainder we express it as a sum, noting that we must include the mean, which
is by symmetry,
\begin{equation}
\mu_\epsilon := \frac{2}{2\pi} \int_0^\epsilon \log \big(2\lb 1 - \cos \psi \rb\big) d \psi = \frac{2}{\pi} \lb \epsilon \log \epsilon - \epsilon +o(\epsilon) \rb.
\end{equation}
Define $S_\epsilon(\varphi) := \lc \theta_q: \theta_q \in [\varphi-\epsilon,\varphi+\epsilon] \rc$ and hence the sum can be written as
\begin{equation}
Z_{N,\epsilon}(\varphi_r) := \frac{1}{\sqrt{s_N}}
\sum_{\theta_q \in S_\epsilon(\varphi_r)}
 \log \big( 2(1 - \cos(\varphi_r-\theta_q))\big) - \frac{1}{\sqrt{s_N}} \mu_\epsilon.
\label{eqn_Zvarepsdef}
\end{equation}
Denote the corresponding sequence as ${\bf Z}_{N,\epsilon}$. Taking expectations we thus find that
$$
\expect{Z_{N,\epsilon}(\varphi_r)} = \frac{1}{\sqrt{s_N}} \int_{-\epsilon}^\epsilon \log \big( 2 (1-\cos\psi) \big) d\psi - \frac{1}{\sqrt{s_N}} \mu_\epsilon = 0
$$
is a sequence of random variables with 0 mean. We finally write,
\begin{equation}
{\bf T}_N = {\bf T}_{N,\epsilon} + {\bf Z}_{N,\epsilon}.
\label{eqn_seqdiff}
\end{equation}
We now proceed to demonstrate the first condition of (\ref{eqn_weakcond}), namely that, ${\bf T}_{N, \epsilon} \Rightarrow  {\bf I}_\epsilon$. The random variable $T_{N, \epsilon}(\varphi_r)$ is a functional of an empirical distribution and therefore of a process
lying in $D[0,2\pi]$. Define the random sequence $J_\epsilon$ defined for $f \in D[0,2\pi]$ and $f(0) = f(2\pi) = 0$ with the component term,
\begin{equation}
J_{\epsilon,r}(f) = \int_\epsilon^{2\pi - \epsilon} f_{\varphi_r}(\psi) 
  \frac{\sin \psi}{\lb 1 - \cos \psi \rb} d \psi -
\Big( f_{\varphi_r}(\psi) \log \big(2(1 - \cos \psi)\big) \Big)_\epsilon^{2\pi-\epsilon}.
\label{eqn_Imap}
\end{equation} 
It is well known that $W_{N,0} \Rightarrow \Wc$ in $D$, which implies that $W_{N,\varphi_r} \Rightarrow \Wc_{\varphi_r}$ as $N\to\infty$ for all $r$. The result follows on showing that $J_\epsilon$  defines a measurable mapping $J_\epsilon:D[0,2\pi] \rightarrow \re^\infty$ in $D[0,2\pi]$. Since 
$$
J_{\epsilon,r}(W_N) = T_{N,\epsilon}(\varphi_r),
$$ 
we may therefore apply Theorem 5.1, Corollary 1 of \cite{Billingsley68} which states that if $W_N \Rightarrow \Wc$ then 
$J_\epsilon(W_N) \Rightarrow J_\epsilon(\Wc)$, (and hence ${\bf T}_{N,\epsilon} \Rightarrow {\bf I}_\epsilon$) provided that we verify 
\begin{equation}
\prob{ \Wc \in D_{J_\epsilon}} = 0.
\label{eqn_probdiscontinuity}
\end{equation}
To deal with the measurability question we first observe that the coordinate maps are measurable and since $\sin \psi/(1-\cos \psi)$ is continuous in $[\epsilon, 2\pi - \epsilon]$, it follows by Lemma \ref{lem_Dctsmap} that $J_{\epsilon,r}$ is measurable for each $r$ and hence so is the sequence mapping $J_\epsilon$. Again by Lemma \ref{lem_Dctsmap} the sequence of integrals convergences with respect to $\rho_0$. This leaves
only the final term. However, since the limit $\Wc$ is almost surely continuous it follows that
\begin{eqnarray*}
f_{\varphi_r}(\epsilon) & \rightarrow & \Wc_{\varphi_r}(\epsilon) \\
f_{\varphi_r}(2\pi - \epsilon) & \rightarrow & \Wc_{\varphi_r}(2\pi-\epsilon) 
\end{eqnarray*}
for each $r$ if $f \rightarrow \Wc$ in $D[0,2\pi]$. Thus the corresponding sequence converges 
with respect to $\rho_0$ also and so (\ref{eqn_probdiscontinuity}) holds. The proof of the first condition is concluded. 

It remains to demonstrate (\ref{eqn_probmetric}). Here we use the union bound and Chebyshev's inequality. This is because the various $Z_{N,\epsilon}(\varphi_r)$ in the sequences are dependent, as
they are determined via the same $\theta_q$. Nevertheless they are of course themselves the sum of i.i.d. 
random variables. In determining the variance,
we may work with $\varphi_r=0$ without loss of generality. The variance
of one of the i.i.d. summands in (\ref{eqn_Zvarepsdef}) is determined as
\begin{equation}
\sigma^2_\epsilon := \frac{2}{2\pi}  \int_0^\epsilon  \log^2 \big(2(1 - \cos\psi)\big)  d\psi - \mu_\epsilon^2 < \infty.
\end{equation}

Since for small $\epsilon > 0$ we have $\log \big(2(1 - \cos\psi)\big) = O(2\log \psi) + o(\psi)$ the integral
is $\sigma_\epsilon = O(\epsilon \log^2 \epsilon)$ as the integral of $\log^2 x$ is $x \log^2 x - 2x \log x + 2x$.
It follows that $\sigma^2_\epsilon \rightarrow 0$ as $\epsilon \rightarrow 0$, which is the variance of
the entire sum by independence and as it has been scaled.  

Now fix $\eta >0$. By definition of $\rho_0$ and from (\ref{eqn_seqdiff}) we obtain that,
$$
\rho_0({\bf T}_{N,\epsilon} , {\bf T}_N ) = \sum_{r=0}^\infty \frac{\abs{Z_{N,\epsilon}(\varphi_r)}}{1 + 
\abs{Z_{N,\epsilon}(\varphi_r)}} 2^{-r}.
$$
Let $R_\eta$ be such that $\sum_{r=R_\eta+1}^\infty 2^{-r} < \eta/2$. Now we apply the union bound to the
remaining $R_\eta+1$ summands to obtain that
\begin{eqnarray}
\prob{\sum_{r=0}^{R_\eta} \frac{\abs{ Z_{N,\epsilon}(\varphi_r)}}{1 + 
\abs{Z_{N,\epsilon}(\varphi_r)}} 2^{-r} \geq \eta/2 } &\leq &\sum_{r=0}^{R_\eta} \prob{\abs{Z_{N,\epsilon}(\varphi_r)} 2^{-r} \geq \frac{\eta}{2 \lb R_\eta +1 \rb}} \\
\label{eqn_chebunionbnd}
& \leq & \sum_{r=0}^{R_\eta} \sigma^2_\epsilon \frac{4 \lb R_\eta +1 \rb ^2}{\eta^2 2^{2r}} \nonumber \\
& \leq & \sigma^2_\epsilon \frac{16 \lb  R_\eta +1 \rb ^2}{3\eta^2}. \nonumber
\end{eqnarray}
Hence,
$$
\limsup_{N\to\infty} \prob{\rho_0({\bf T}_{N,\epsilon} , {\bf T}_N ) > \eta } \leq \frac{16(R_\eta + 1)^2 \sigma^2_\epsilon}{3\eta^2}
$$
and the RHS goes to 0 as $\epsilon$ to 0, for each $\eta > 0$. Hence we obtain (\ref{eqn_probmetric}) as 
required.  Therefore, we have verified all conditions and Theorem \ref{thm_weakconvergence} is proved.
\end{proof}

It is rather easy to see that $I^* > 0$ almost surely. For instance, as in the proof of Lemma \ref{lemma_ctsI} it can be shown
that,
$$
\int_0^{2\pi} \int_0^{2\pi} \Big\vert W^o_{\varphi}(\psi) \frac{\sin \psi}{1-\cos \psi} \Big\vert \,d\psi \,d\varphi < \infty.
$$
It then follows from Fubini's theorem that, 
$$
\int_0^{2\pi} I_\varphi \,d\varphi = 0
$$ 
since
$\int_0^{2\pi} W^o_\varphi(\psi) d\varphi = 0$. But $I_\varphi$ is almost surely continuous and hence
$I^{*}=0$ if and only if $I_\varphi = 0$ for every $\varphi$. It thus follows that $I^* > 0$ almost
surely, as required.

\begin{theorem}
\label{main_thm}
Given $\tau \geq 0$, 
$$
\liminf_{N \rightarrow \infty} \prob{ T_N^* > \tau } \geq \prob{I^* > \tau}
$$
with equality if $\tau$ is a continuity point for the random variable $I^*$. Moreover, 
\begin{equation}
\frac{1}{s_N} \log \max \Big\{ |P_N(z)|^2\,\,:\,\,|z|=1\Big\} \Rightarrow I^{*}.
\end{equation}
\end{theorem}

It therefore follows that,
$$
\max \Big\{ |P_N(z)|^2\,\,:\,\,|z|=1\Big\} \approx e^{\sqrt{n_1^{2}+\ldots+n_{N}^{2}}I^{*}} 
$$
for $N$ sufficiently large. A sample histogram for $I^*$ is shown in Figure \ref{fig_I*}.
\begin{figure}[!Ht]
  \begin{center}
    \includegraphics[width=12cm]{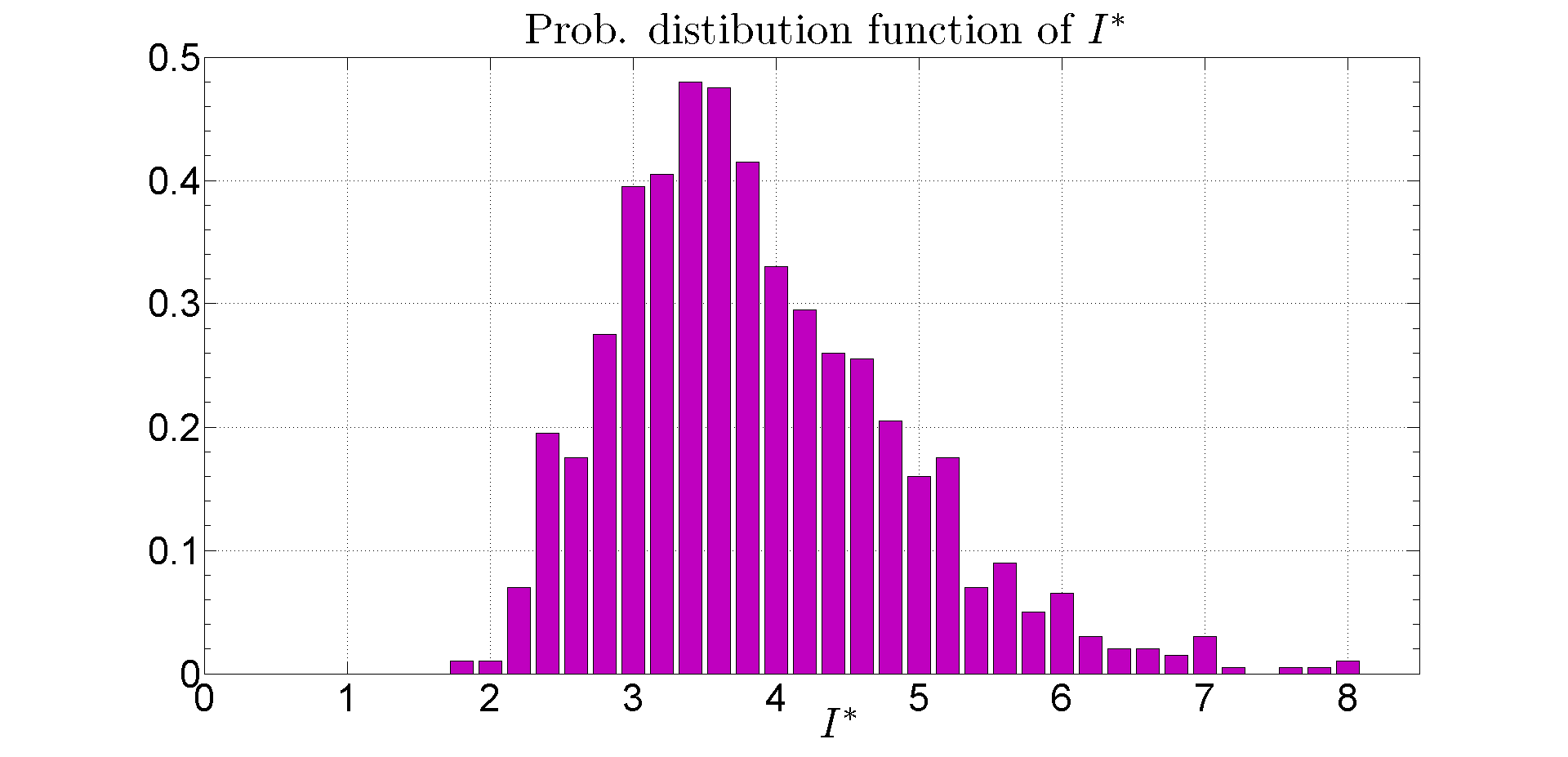}
    \caption{Histogram of the pdf of $I^{*}$.}
    \label{fig_I*}
  \end{center}
\end{figure}

\section{Numerical Results}\label{sec_numer}
In this Section, we present some simulations of our results. In Figure \ref{fig_randpoly} we show the logarithm squared magnitude for a random polynomial with $N=500$ and constant sequence $n_k=1$. Here the maximum value is $\approx 10^{60}$ and occurs near $\psi = 0.3$.
\begin{figure}[!Ht]
  \begin{center}
    \includegraphics[width=12cm]{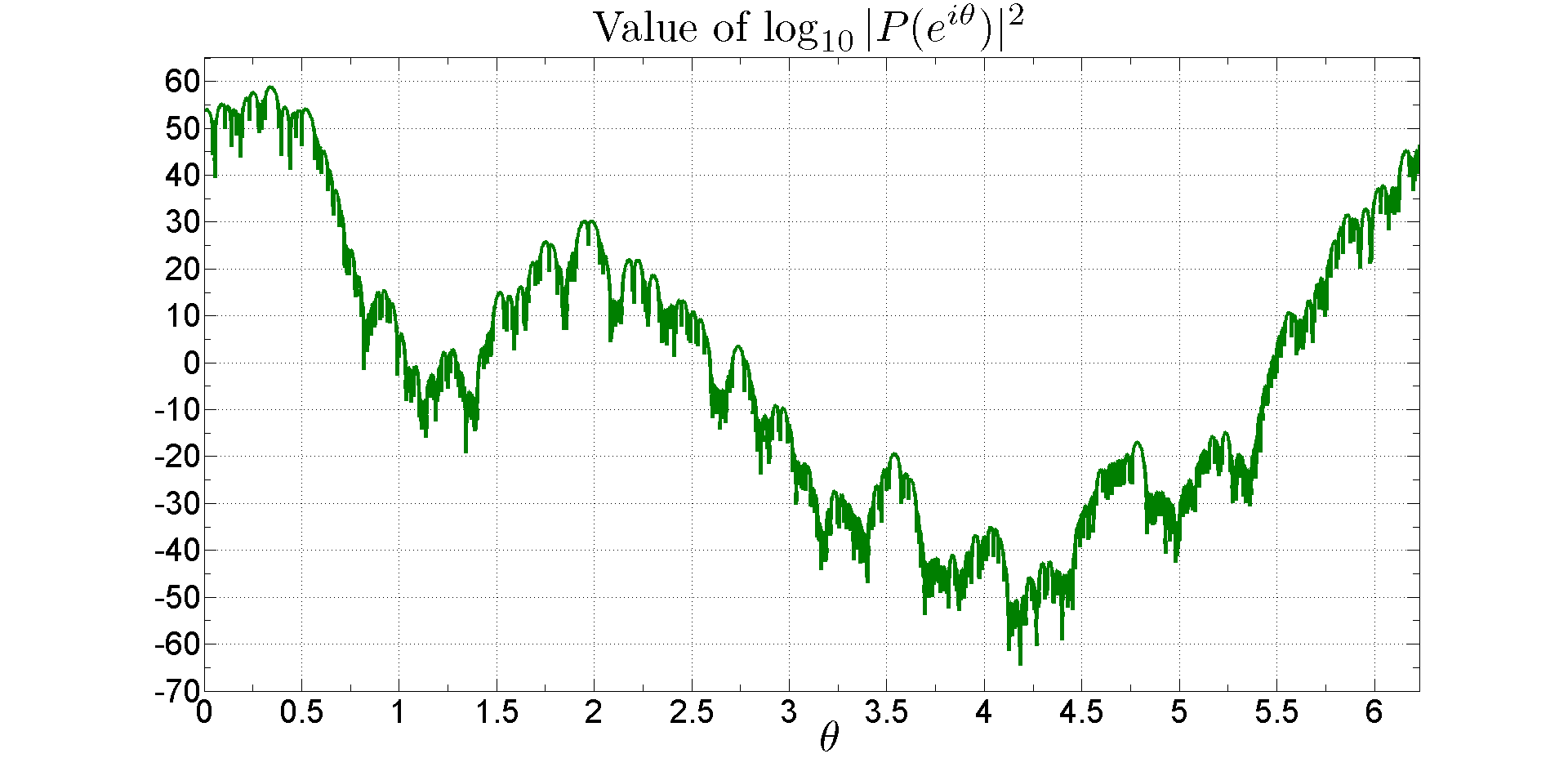}
    \caption{Log magnitude squared as a function of the phase ($N=500$ and $n_{k}=1$).}
    \label{fig_randpoly}
  \end{center}
\end{figure}
Our final plots show the logarithm of the maximum magnitude as a function of the degree for the sequences $n_{k}=1$ and the sequence $n_k=k$ with 100 realizations per degree.
The blue curves are $(n_1^2+\ldots+n_N^2)^{1/2}$ and $5(n_1^2+\ldots+n_N^2)^{1/2}$ respectively for both cases.
\begin{figure}[!Ht] 
  \begin{center}
    \includegraphics[width=12cm]{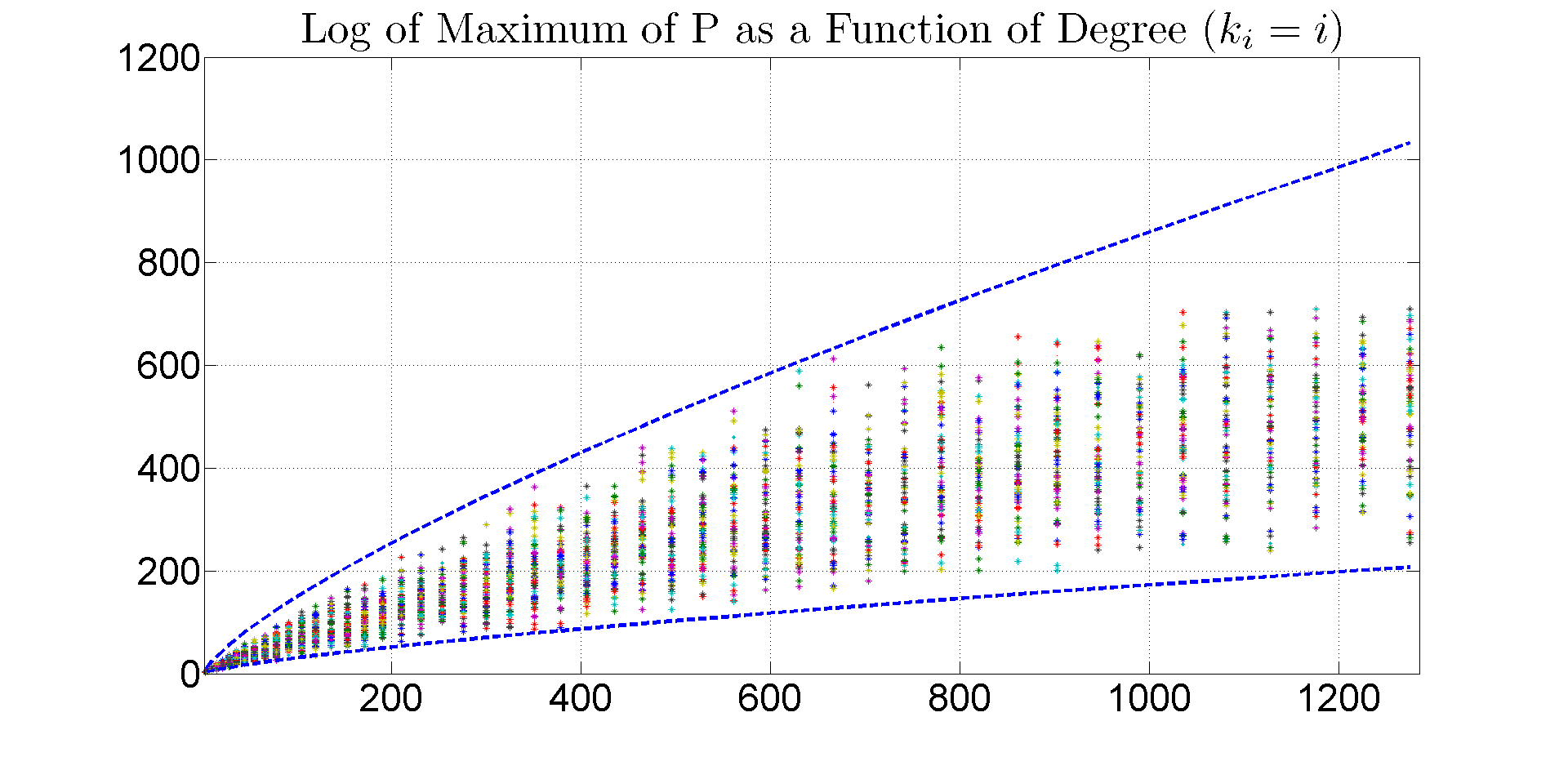}
    \includegraphics[width=12cm]{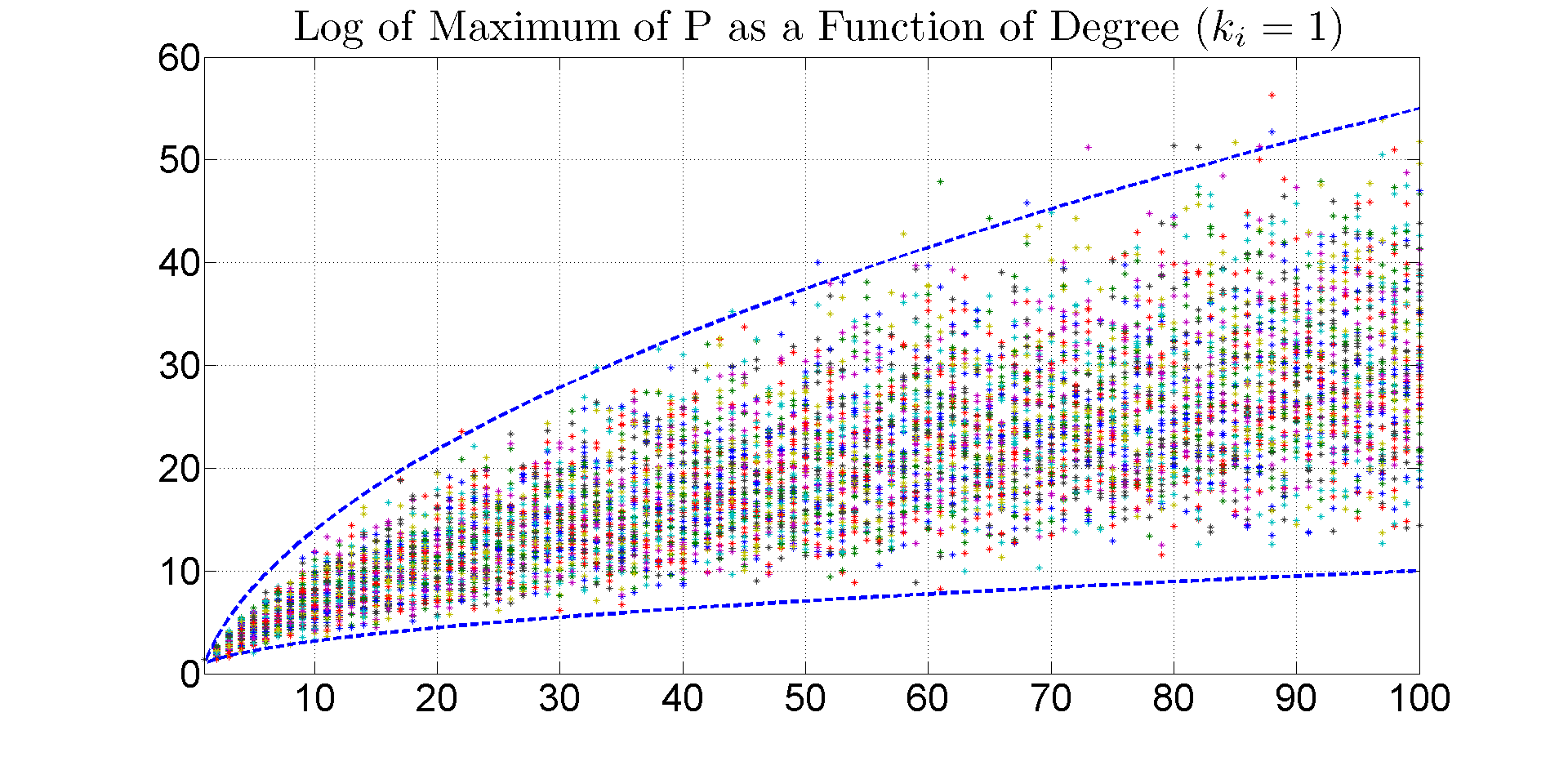}
    \caption{Logarithm of the maximum magnitude as a function of the degree for the sequences $n_{k}=k$ (top) and the sequence $n_k=1$ (bottom) and 100 realizations per time.
    The blue curves are $(n_1^2+\ldots+n_N^2)^{1/2}$ and $5(n_1^2+\ldots+n_N^2)^{1/2}$ for both cases.}
    \label{fig_degree}
  \end{center}
\end{figure}

\end{document}